\documentclass[12pt]{amsart}%
\usepackage{amsfonts}
\usepackage{txfonts}
\usepackage{mathtools}
\usepackage{amsmath}
\usepackage{amssymb}
\usepackage{amsthm}
\usepackage{newlfont}
\usepackage{graphicx}%
\providecommand{\U}[1]{\protect\rule{.1in}{.1in}}
\newtheorem{thm}{Theorem}[section]

\newtheorem{proposition}{Theorem}[section]

\newtheorem{cor}[thm]{Corollary}
\newtheorem{lem}[thm]{Lemma}
\newtheorem{prop}[thm]{Proposition}

\theoremstyle{definition}

\theoremstyle{remark}
\newtheorem{rem}[thm]{Remark}
\numberwithin{equation}{section}

\newcommand{\ed}{\end {document}}

\begin{document}
\title[Radial solutions]{Radial solutions of a fourth order Hamiltonian stationary equation}

\author{Jingyi Chen}
\author{Micah Warren}

\address{Department of Mathematics\\
The University of British Columbia, Vancouver, BC V6T1Z2, Canada}
\email{jychen@math.ubc.ca}

\address{Department of Mathematics\\
University of Oregon, Eugene, OR 97403, U.S.A.}
\email{micahw@uoregon.edu}

\thanks{The first author was supported in by NSERC Discovery Grant 22R80062. The
second author was partially supported by NSF Grant DMS-1438359. }

\begin{abstract}
We consider smooth radial solutions to the Hamiltonian stationary equation
which are defined away from the origin. We show that in dimension two all
radial solutions on unbounded domains must be special Lagrangian. In contrast,
for all higher dimensions there exist non-special Lagrangian radial solutions
over unbounded domains; moreover, near the origin, the gradient graph of such
a solution is continuous if and only if the graph is special Lagrangian.

\end{abstract}
\maketitle

\section{Introduction}

An important class of Lagrangian submanifolds in a symplectic manifold is the
so-called Hamiltonian stationary submanifolds, which are Lagrangian
submanifolds and are critical points of the volume functional under
Hamiltonian variations. A well-known subset of this class consists of the
special Lagrangian submanifolds. \ These are Lagrangian and critical for the
volume functional for all variations; so they are minimal Lagrangian
submanifolds. The special Lagrangians form one of the most distinguished
classes in calibrated geometry, and they play a unique role in string theory.
In this paper, we shall study the radially symmetric Hamiltonian stationary
graphs in the complex Euclidean space $\mathbb{C}^{n}$, and explore conditions
under which the Hamiltonian stationary ones reduce to the special Lagrangians.

For a fixed bounded domain $\Omega\subset{\mathbb{R}}^{n}$, let $u:\Omega
\rightarrow\mathbb{R}$ be a smooth function. The gradient graph $\Gamma
_{u}=\left\{  \left(  x,Du(x)\right)  :x\in\Omega\right\}  $ is a Lagrangian
$n$-dimensional submanifold in $\mathbb{C}^{n}$, with respect to the complex
structure $J$ defined by $z_{j}=x_{j}+iy_{j}$ for $j=1,\cdots,n$. The volume
of $\Gamma_{u}$ is given by
\[
F_{\Omega}(u)=\int_{\Omega}\sqrt{\det\left(  I+\left(  D^{2}u\right)
^{T}D^{2}u\right)  }dx.
\]
A smooth function $u$ is critical for the volume functional $F_{\Omega}(u)$
under compactly supported variations of the scalar function if and only if $u$
satisfies the equation
\begin{equation}
\Delta_{g}\theta=0 \label{Hamstat}%
\end{equation}
where $\Delta_{g}$ is the Laplace-Beltrami operator on $\Gamma_{u}$ for the
induced metric $g$ from the Euclidean metric on $\mathbb{R}^{2n}$, (c.f.
\cite{MR1202805}, \cite[Proposition 2.2]{SW}). Here, the Lagrangian phase
function is defined by
\[
\theta=\operatorname{Im}\log\det\left(  I_{n}+\sqrt{-1}D^{2}u\right)
\]
or equivalently,
\begin{equation}
\theta=\sum_{i=1}^{n}\arctan\lambda_{i} \label{thetadef2}%
\end{equation}
for $\lambda_{i}$ the eigenvalues of $D^{2}u.$ The mean curvature vector along
$\Gamma_{u}$ can be written
\[
\vec{H}=J\nabla\theta
\]
where $\nabla$ is the gradient operator of $\Gamma_{u}$ for the metric $g$
(cf. \cite[2.19]{HL}).

The gradient graph of $u$ which solves (\ref{Hamstat}) is called Hamiltonian
stationary. A Hamiltonian stationary gradient graph $\Gamma_{u}$ is a critical
point of the volume functional $F_{\Omega}(\cdot)$ under Hamiltonian
deformations, that is, those generated by $J\nabla\eta$ for some smooth
compactly supported function $\eta$. On the other hand, recall that if $u$
satisfies the special Lagrangian equation \cite{HL}
\begin{equation}
\nabla\theta=0 \label{sLag}%
\end{equation}
i.e. $\vec{H}\equiv0$, then the surface is critical for the volume functional
under \textit{all} compactly supported variations of the surface $\Gamma_{u}$.
In terms of the potential function $u$, the Hamiltonian stationary equation is
of fourth order and the special Lagrangian equation is of second order, both
are elliptic. There are Hamiltonian stationary but not special Lagrangian
surfaces even when $n=2$, and this causes serious problems for constructing
special Lagrangian surfaces (see \cite{SW}).

In this paper, we consider radial solutions of the Hamiltonian stationary
equation (\ref{Hamstat}) on a domain which may not contain the origin. Our
first observation is that the fourth order Hamiltonian stationary equation
reduces to the second order special Lagrangian equation for radial solutions
defined on any unbounded domain in $\mathbb{R}^{2}$. However, this is not the
case for $n>2$.

We show the following:

\begin{thm}
\label{one}Suppose that $u=u\left(  r\right)  $ is a radial solution of the
fourth order Hamiltonian stationary equation (\ref{Hamstat}) on$~\left\{
x\in\mathbb{R}^{2}:\left\vert x\right\vert >a\right\}  $ for some $a\geq0.$
Then $u$ is a solution of the special Lagrangian equation (\ref{sLag}). For
any $n>2,$ (\ref{Hamstat}) admits a radial solution over any unbounded domain
in $\mathbb{R}^{n}\backslash\{0\}$, which is not special Lagrangian.
\end{thm}

Theorem \ref{one} will be proved in two propositions: Proposition \ref{prop1}
\ and Proposition \ref{prop41}. \ 

Radial solutions can be characterized according to their behavior near 0. For
$n=3$, we exhibit a strong solution (a classical solution away from the origin
of $\mathbb{R}^{3}$) the closure of whose gradient graph over the unit 3-ball
$\mathbb{B}^{3}$ is a smoothly embedded submanifold with boundary that is
diffeomorphic to a half cylinder, yet it is not Hamiltonian isotopic to the
gradient graph of any smooth function over $\mathbb{B}^{3}$. In fact, we are
able to show the following:

\begin{thm}
\label{two}Suppose that $u=u\left(  r\right)  $ is a radial solution of the
fourth order Hamiltonian stationary equation (\ref{Hamstat}) on$~\left\{
x\in\mathbb{R}^{n}:0<|x|<b\right\}  .$ If%
\[
\lim_{r\rightarrow0}u'(r)=0
\]
then $u$ is a solution of the special Lagrangian equation and the completed
graph is a flat disk$.$

If%
\[
\lim_{r\rightarrow0}u^{\prime}(r)=c\neq0
\]
then the gradient graph (including the sphere over the origin as its boundary)
is a properly embedded half cylinder.
\end{thm}

Note that, due to the strong maximum principle applied to (\ref{Hamstat}) on
any $B_{R}(0),$ any radial and harmonic $\theta$ that extends to a weakly
harmonic function across the origin must be constant. \ It follows from a
removable singularity result of Serrin \cite{Serrin} that if a radial solution
$u$ is $C^{1,1}$ near the origin, then $\theta$ must extend to a weakly
harmonic function and must be constant (cf. \cite{CWReg}). The corresponding
potential function $u$ must be quadratic, as we will see in section 5 where we
investigate properties of the ODE and combine them with a calibration argument.

Bernstein type results for special Lagrangian equations in dimension 2 were
obtained by Fu \cite{Fu}: the global solutions are either quadratic
polynomials or harmonic functions. For convex solutions and for large phases
in higher dimensions Bernstein results were given by Yuan \cite{YYI,YY06}. For
Bernstein and Liouville results for (\ref{Hamstat}) with constraints, and more
discussion of the problem, see \cite{Mese}, \cite{Warren2016}.

The rest of the paper is organized as follows. In section 2, we will derive an
ODE that characterizes radial solutions, and prove that the Dirichlet integral
of the phase function is always finite in the radially symmetric case. In
section 3, we show that for $n=2$ all radial solutions on unbounded domains
must be special\ Lagrangian, and write down all radial solutions to special
Lagrangian equations. In section 4, we show that, for $n>2$, non-special
Lagrangian solutions to (\ref{Hamstat}) exist on $\mathbb{R}^{n}%
\backslash\left\{  0\right\}  .$ In section 5, we explore the behavior near
the origin, and show that continuity of the gradient graph implies a radial
solution is special Lagrangian.

\section{The radial Hamiltonian stationary equation}

We are interested in the gradient graph $\Gamma=\left\{  \left(
x,Du(x)\right)  :x\in\mathbb{R}^{n}\backslash\{0\}\right\}  $ of a radially
symmetric function $u$. Suppose $u=u(r)$, then
\[
Du=u^{\prime}(r)Dr.
\]
In terms of the standard parametrization
\begin{align*}
(0,\infty)\times\mathbb{S}^{n-1} &  \rightarrow\mathbb{R}^{n}\backslash\{0\}\\
(r,\xi) &  \rightarrow r\xi
\end{align*}
we can parametrize the gradient graph as
\begin{equation}
(r,\xi)\rightarrow(r\xi,u^{\prime}(r)\,\xi)\in\mathbb{R}^{n}\times
\mathbb{R}^{n}.\label{embedding}%
\end{equation}
Define the following notations for tangent vectors at a point $(r\xi
,u^{\prime}(r)\xi)$ of $\Gamma$
\[
\partial_{r}=(\xi,u^{\prime\prime}(r)\,\xi)\in T_{\left(  r\xi,u^{\prime
}(r)\,\xi\right)  }\Gamma
\]
and
\[
\partial_{V}=(rV,u^{\prime}(r)V)\in T_{\left(  r\xi,u^{\prime}(r)\,\xi\right)
}\Gamma
\]
for each $V\in T_{\xi}\mathbb{S}^{n-1}$. Thus, the induced metric $g$ on
$\Gamma$ may be computed as
\begin{align*}
g_{rr} &  =\langle\partial_{r},\partial_{r}\rangle=1+\left(  u^{\prime\prime
}\right)  ^{2}\\
g_{VV} &  =\langle\partial_{V},\partial_{V}\rangle=r^{2}+\left(  u^{\prime
}\right)  ^{2}\\
g_{V_{i}V_{j}} &  =\langle\partial_{V_{i}},\partial_{V_{j}}\rangle
=0,\,\,\,\,{\forall i\not =j,V_{i}\perp V_{j}},V_{i},V_{j}\in T_{\xi
}\mathbb{S}^{n-1}\\
g_{rV} &  =\langle\partial_{r},\partial_{V}\rangle=0.
\end{align*}
In the $(r,\xi)$-coordinates, the Hamiltonian stationary equation
\[
\sum_{i,j=1}^{n}\frac{1}{\sqrt{g}}\left(  \sqrt{g}g^{ij}\theta_{i}\right)
_{j}=0
\]
reduces to (because only $r$ derivatives persist)
\[
\left(  \sqrt{g}g^{rr}\theta_{r}\right)  _{r}=0.
\]
This directly gives that
\[
\sqrt{g}g^{rr}\theta_{r}=C,
\]
in particular%
\begin{equation}
\theta_{r}=\frac{Cg_{rr}}{\sqrt{g}}=C\sqrt{\frac{1+\left(  u^{\prime\prime
}\right)  ^{2}}{\left(  r^{2}+\left(  u^{\prime}\right)  ^{2}\right)  ^{n-1}%
}.}\label{hs2}%
\end{equation}
To compute $\theta,$ we may diagonalize $D^{2}u$ with respect to the Euclidean
coordinates. Taking $x_{1}=r,x_{2}=0,\dots,x_{n}=0$, then
\[
D^{2}u=\left(
\begin{array}
[c]{cccc}%
u^{\prime\prime} & 0 & ... & 0\\
0 & \frac{u^{\prime}}{r} & 0 & ...\\
... & 0 & ... & 0\\
0 & ... & 0 & \frac{u^{\prime}}{r}%
\end{array}
\right)  .
\]
Thus by (\ref{thetadef2})
\[
\theta=\arctan\left(  u^{\prime\prime}\right)  +\left(  n-1\right)
\arctan\left(  \frac{u^{\prime}}{r}\right)
\]
and
\[
\theta_{r}=\frac{1}{1+\left(  u^{\prime\prime}\right)  ^{2}}u^{\prime
\prime\prime}+\left(  n-1\right)  \frac{1}{1+\left(  \frac{u^{\prime}}%
{r}\right)  ^{2}}\left(  \frac{u^{\prime}}{r}\right)  ^{\prime}.
\]
Finally, in terms of the radial function $u$, the Hamiltonian stationary
equation (\ref{hs2}) becomes
\begin{equation}
\frac{1}{1+\left(  u^{\prime\prime}\right)  ^{2}}u^{\prime\prime\prime
}+\left(  n-1\right)  \frac{1}{1+\left(  \frac{u^{\prime}}{r}\right)  ^{2}%
}\left(  \frac{u^{\prime}}{r}\right)  ^{\prime}=C\sqrt{\frac{1+\left(
u^{\prime\prime}\right)  ^{2}}{\left(  r^{2}+\left(  u^{\prime}\right)
^{2}\right)  ^{n-1}}}.\label{hs3}%
\end{equation}
Note that $C=0$ corresponds to the special Lagrangian equation \eqref{sLag}.

Letting
\[
v=u^{\prime},
\]
the equation (\ref{hs3}) becomes a second order ODE in $v$:%

\begin{equation}
\frac{1}{1+\left(  v^{\prime}\right)  ^{2}}v^{\prime\prime}+\left(
n-1\right)  \frac{v^{\prime}r-v}{r^{2}+v^{2}}-C\frac{\left(  1+\left(
v^{\prime}\right)  ^{2}\right)  ^{\frac{1}{2}}}{\left(  r^{2}+v^{2}\right)
^{\frac{n-1}{2}}}=0. \label{main}%
\end{equation}

Next, we note a consequence of (\ref{hs2}): For any radial solution to the
Hamiltonian stationary equation, the energy of the Lagrangian phase function
$\theta$  admits a uniform upper bound, independent of the size of the domain.

\begin{proposition}
\label{Efinite} Suppose that $u\left(  r\right)  $ is a radial solution of the
fourth order Hamiltonian stationary equation (\ref{Hamstat}) on$~\left\{
x\in\mathbb{R}^{n}:a<\left\vert x\right\vert <b.\right\}  $. Letting
$\Gamma=\left\{  \left(  x,Du(x)\right)  :a<\left\vert x\right\vert
<b\right\}  $, then
\[
\int_{\Gamma}\left\vert \nabla_{g}\theta\right\vert ^{2}dV_{g}=C \, \emph{Vol}
( \mathbb{S}^{n-1}) \left[  \theta(b)-\theta(a)\right]  ,
\]
where $C$ is the constant in (\ref{hs2}). In particular, the energy of
$\theta$ is always finite.
\end{proposition}

\begin{proof}
Using the fact that $\theta$ is radial,
\[
\left\vert \nabla_{g}\theta\right\vert ^{2}=g^{rr}\theta_{r}^{2}.
\]
Now using (\ref{hs2}),
\[
\left\vert \nabla_{g}\theta\right\vert ^{2}= g^{rr}\theta_{r}\theta_{r}%
=g^{rr}\frac{C^{2}g_{rr}^{2}}{g}=\frac{C^{2}g_{rr}}{g}.
\]
Thus
\begin{align*}
\int_{\mathbb{B}_{b}^{n}(0)\backslash\mathbb{B}_{a}^{n}(0)}\left\vert
\nabla_{g}\theta\right\vert ^{2}dV_{g}  &  =\int_{\mathbb{S}^{n-1}}\int%
_{a}^{b}\frac{C^{2}g_{rr}(r)}{g(r)}\sqrt{g(r)}drdV_{S^{n-1}}\\
&  =\mbox{Vol}(\mathbb{S}^{n-1})\int_{a}^{b}\frac{C^{2}g_{rr}(r)}{\sqrt{g(r)}%
}dr.
\end{align*}
On the other hand, by \eqref{hs2} again
\[
\theta(b)-\theta(a)=\int_{a}^{b}\theta_{r}dr=\int_{a}^{b}\frac{Cg_{rr}%
(r)}{\sqrt{g(r)}}dr.
\]
Putting these together finishes the proof.
\end{proof}

\section{Reduction of order for $n=2$}

The equation \eqref{main} always admits a short time solution near $0$ for any
constant $C,$ provided that $v(0)\neq0.$ In particular, the solutions for
$C\not =0$ are not solutions to the special Lagrangian equation for any
$n\geq2$. However, when $n=2$, we will show that the solutions to \eqref{main}
which are defined on any \textit{unbounded} domain in $\mathbb{R}^{2}$ are
necessarily special Lagrangian. In the next section, we will exhibit existence
of solutions on $\mathbb{R}^{n}\backslash\{0\}$ for $n>2$ which are not
special Lagrangian. So the reduction of order for radial solutions, namely
from the fourth order equation \eqref{main} to the second order equation
\eqref{sLag} only happens in dimension 2 and is non-local.

We first investigate the case that the limit of $\theta$ at infinity is not
zero. \ 

\begin{prop}
\label{easy proof} When $n=2$, if $u=u(r)$ is a radial solution on
$(a,\infty)$ to (\ref{hs3}) with
\[
\lim_{r\rightarrow\infty}\theta(r)\neq0
\]
then $C=0$. In particular, $u$ is a solution of the special Lagrangian equation.
\end{prop}

\begin{proof}
Recall that $\theta$ is bounded and monotone by (\ref{hs2}), so $\theta$ has a
finite limit when $r\rightarrow\infty$. Without loss of generality we may take
$\lim_{r\rightarrow\infty}\theta(r)=\delta>0.$ \ In particular, we may choose
$R_{0}$ large enough so that
\[
\theta(r)>\frac{\delta}{2},\ \text{ for }r>R_{0}.
\]
Letting%
\[
\Omega=\mathbb{R}^{2}\backslash\mathbb{B}_{R_{0}}^{2}(0)\subset\mathbb{R}%
^{2}\backslash\mathbb{B}_{a}^{2}(0)
\]
we consider the portion of the gradient graph of $u$ restricted to $\Omega,$
namely
\[
\Gamma=\left\{  \left(  x,Du(x)\right)  :x\in\mathbb{R}^{2}\backslash
\mathbb{B}_{R_{0}}^{2}(0)\right\}  .
\]
Now for $\lambda_{1},\lambda_{2},$ the eigenvalues of $D^{2}u$, we have
\[
\arctan\lambda_{1}+\arctan\lambda_{2}>\frac{\delta}{2}.
\]
So
\[
D^{2}u(x)>\tan\left(  \frac{\delta}{2}-\frac{\pi}{2}\right)  >-\infty
\]
when $x\in\Omega$. Now we may apply the Lewy-Yuan rotation argument in
\cite[Step 1 in section 2]{YY06}, (cf. \cite[Proposition 4.1]{CWReg}). \ In
particular, there is a domain $\bar{\Omega}\subset\mathbb{R}^{n},$ and a
function $\bar{u}$ on $\bar{\Omega}$ such that the surface $\bar{\Gamma}$
$\subset\mathbb{R}^{2}+\sqrt{-1}\mathbb{R}^{2}$ defined by
\[
\bar{\Gamma}=\left\{  \left(  \bar{x},D\bar{u}(\bar{x}\right)  ):\bar{x}%
\in\bar{\Omega}\right\}
\]
is isometric to $\Gamma$ via a unitary rotation of $%
\mathbb{C}
^{2}.$ \ Expressly, letting
\begin{equation}
\bar{x}(x)=\cos\left(  \frac{\delta}{4}\right)  x+\sin\left(  \frac{\delta}%
{4}\right)  Du(x)\label{ff1}%
\end{equation}
we have%
\[
\bar{\Omega}=\left\{  \bar{x}(x):x\in\mathbb{R}^{2}\backslash\mathbb{B}%
_{R_{0}}^{2}(0)\right\}  .
\]
Also, when $\bar{x}=\bar{x}(x),$
\begin{equation}
D\bar{u}\left(  \bar{x}\right)  =-\sin\left(  \frac{\delta}{4}\right)
x+\cos\left(  \frac{\delta}{4}\right)  Du(x).\label{ff2}%
\end{equation}
When $u$ is radial, the derivative satisfies
\[
Du(x)=u^{\prime}(\left\vert x\right\vert )\frac{x}{\left\vert x\right\vert }%
\]
so (\ref{ff1}) becomes%
\begin{align}
\bar{x}(x) &  =\cos\left(  \frac{\delta}{4}\right)  x+\sin\left(  \frac
{\delta}{4}\right)  u^{\prime}(\left\vert x\right\vert )\frac{x}{\left\vert
x\right\vert }\\
&  =\left(  \cos\left(  \frac{\delta}{4}\right)  +\sin\left(  \frac{\delta}%
{4}\right)  \frac{u^{\prime}(\left\vert x\right\vert )}{\left\vert
x\right\vert }\right)  x
\end{align}
and (\ref{ff2}) becomes%
\begin{align}
D\bar{u}\left(  \bar{x}\right)   &  =\left(  -\sin\left(  \frac{\delta}%
{4}\right)  +\cos\left(  \frac{\delta}{4}\right)  \frac{u^{\prime}(\left\vert
x\right\vert )}{\left\vert x\right\vert }\right)  x\\
&  =\left(  -\sin\left(  \frac{\delta}{4}\right)  +\cos\left(  \frac{\delta
}{4}\right)  \frac{u^{\prime}(\left\vert x\right\vert )}{\left\vert
x\right\vert }\right)  \left(  \cos\left(  \frac{\delta}{4}\right)
+\sin\left(  \frac{\delta}{4}\right)  \frac{u^{\prime}(\left\vert x\right\vert
)}{\left\vert x\right\vert }\right)  ^{-1}\bar{x}.
\end{align}
We conclude that $D\bar{u}(\bar{x})$ is a multiple of $\bar{x},$ so the
function $\bar{u}$ is radial in the $\bar{x}$-coordinates.

Under this representation, we have (cf. \cite[Proposition 4.1]{CWReg}) that
the inverse map $\bar{x}^{-1}$ exists and is Lipschitz, that is
\[
D\bar{x}\geq c_{\delta}I_{2}>0.
\]
It follows that the complement of $\bar{\Omega}$ must be contained in a
compact set: \
\[
\mathbb{R}^{2}\backslash\bar{\Omega}\subset\mathbb{B}_{\bar{R}_{0}}^{2}(0)
\]
for some $\bar{R}_{0}$. Further, we have that
\[
\tan\left(  \frac{\delta}{2}-\frac{\pi}{2}-\frac{\delta}{4}\right)  <D^{2}%
\bar{u}<\tan\left(  \frac{\pi}{2}-\frac{\delta}{4}\right)  .
\]
The induced metric on the gradient graph of $\bar{u}$ is still $g$ because the
two gradient graphs are isometric. \ On the gradient graph of $\bar{u}$, the
metric $g$ is given in terms of the $\bar{x}$-coordinates by%
\[
g=g_{\bar{\imath}\bar{j}}d\bar{x}^{i}d\bar{x}^{j}%
\]
with
\[
g_{\bar{\imath}\bar{j}}=\delta_{\bar{\imath}\bar{j}}+\bar{u}_{\bar{\imath}%
\bar{k}}\delta^{\bar{k}\bar{l}}\bar{u}_{\bar{l}\bar{j}}.
\]
Thus, in these coordinates we have
\[
d\bar{x}^{2}\leq g\leq C_{1}d\bar{x}^{2}.
\]
The volume form has the expression
\[
dV_{g}=\sqrt{g}\ d\bar{x}_{1}\wedge d\bar{x}_{2},
\]
so for any $R>\bar{R}_{0},$ if we define
\[
\Gamma_{R}:=\left\{  \left(  \bar{x},D\bar{u}(\bar{x}\right)  ):\bar{x}\in
\bar{\Omega},\left\vert \bar{x}\right\vert <R\right\}  \subset\bar{\Gamma},
\]
then we have%
\[
\mbox{Vol}(\Gamma_{R})=\int_{\Gamma_{R}}dV_{g}\leq C_{2}R^{2}.
\]
\ \ Now because $\bar{u}$ is radial, so is the function $\bar{\theta},$ where
\[
\bar{\theta}=\arctan\left(  \bar{\lambda}_{1}\right)  +\arctan\left(
\bar{\lambda}_{2}\right)
\]
for $\bar{\lambda}_{1},\bar{\lambda}_{2}$ eigenvalues of $D^{2}\bar{u}.$ \ One
can check (cf. \cite[pg. 20]{CWReg}) that
\[
\bar{\theta}=\theta-\frac{\delta}{2}%
\]
so $\bar{\theta}$ is still harmonic with respect to the metric $g$ on
$\bar{\Gamma}$ and must satisfy an ordinary differential equation of the form
(\ref{hs2}).

Now choose any interval $\left[  a,b\right]  $ with $a>\bar{R}_{0}$, and some
larger $R$. Taking $\bar{x}$ as coordinates for $\bar{\Gamma}$, we can define
the function
\[
\rho\left(  \bar{x}\right)  =\left\vert \bar{x}\right\vert .
\]
We then define a Lipschitz test function $\eta(\rho\left(  \bar{x}\right)  )$
on $\bar{\Gamma}$ as follows:
\begin{align*}
\eta &  =0\ \ \ \ \ \ \ \ \ \ \ \text{if }\rho<a\\
\eta^{\prime} &  =\frac{1}{b-a}\ \ \ \ \text{if }a\leq\rho\leq b\\
\eta &  =1\ \ \ \ \ \ \ \ \ \ \ \text{if }b\leq\rho\leq b+R\\
\eta^{\prime} &  =-\frac{{1}}{R}\ \ \ \ \ \ \ \text{if }b+R\leq\rho\leq b+2R\\
\eta &  =0\ \ \ \ \ \ \ \ \ \ \ \text{if }\rho>b+2R.
\end{align*}
We integrate (\ref{Hamstat}) by parts and replace $\bar{\theta}$ by
$\theta-\frac{\delta}{2}$:
\[
0=\int_{\bar{\Gamma}}\eta\Delta_{g}\theta\,dV_{g}=-\int_{\bar{\Gamma}}%
\langle\nabla_{g}\theta,\nabla_{g}\eta\rangle dV_{g}.
\]
Thus
\begin{align*}
\int_{\Gamma_{b}\backslash\Gamma_{a}}\langle\nabla_{g}\theta,\nabla_{g}%
\eta\rangle dV_{g} &  =-\int_{\Gamma_{b+2R}\backslash\Gamma_{b+R}}%
\langle\nabla_{g}\theta,\nabla_{g}\eta\rangle dV_{g}\\
&  \leq\left(  \int_{\Gamma_{b+2R}\backslash\Gamma_{b+R}}\left\vert \nabla
_{g}\theta\right\vert ^{2}dV_{g}\right)  ^{1/2}\left(  \int_{\Gamma
_{b+2R}\backslash\Gamma_{b+R}}\left\vert \nabla_{g}\eta\right\vert ^{2}%
dV_{g}\right)  ^{1/2}\\
&  \leq\left(  \int_{\Gamma\backslash\Gamma_{b+R}}\left\vert \nabla_{g}%
\theta\right\vert ^{2}dV_{g}\right)  ^{1/2}\frac{1}{R}\left(
\mbox{Vol}(\text{ }\Gamma_{b+2R}\backslash\Gamma_{b+R})\right)  ^{1/2}\\
&  \leq\left(  \int_{\Gamma\backslash\Gamma_{b+R}}\left\vert \nabla_{g}%
\theta\right\vert ^{2}dV_{g}\right)  ^{1/2}C_{2}.
\end{align*}
Now
\[
\lim_{R\rightarrow\infty}\left(  \int_{\Gamma\backslash\Gamma_{b+R}}\left\vert
\nabla_{g}\theta\right\vert ^{2}dV_{g}\right)  ^{1/2}=0
\]
as the energy of $\theta$ is finite by Theorem \ref{Efinite}. We conclude
that
\[
\int_{\Gamma_{b}\backslash\Gamma_{a}}\langle\nabla_{g}\theta,\nabla_{g}%
\eta\rangle dV_{g}=0.
\]
But%
\[
\int_{\Gamma_{b}\backslash\Gamma_{a}}\langle\nabla_{g}\theta,\nabla_{g}%
\eta\rangle dV_{g}=\frac{1}{b-a}\int_{\Gamma_{b}\backslash\Gamma_{a}}%
g^{\bar{r}\bar{r}}\theta_{\bar{r}}dV_{g}.
\]
Because $\theta$ is radial in $\bar{x}$, and the above two lines are true for
every $b>a>R_{0}$, we conclude that
\[
\theta_{\bar{r}}\equiv0.
\]
It follows that $\theta$ and $\bar{\theta}$ are both constant, and both
$\Gamma$ and  $\bar{\Gamma}$ are special Lagrangian. The potential function
$u$ must be a solution to the special Lagrangian equation. \ 
\end{proof}

In light of Proposition \ref{easy proof}, to show any radial solution to
\eqref{hs3} over $(a,\infty)$ is special Lagrangian, it suffices to consider
the remaining case $\lim_{r\to\infty}\theta(r)=0$.

\begin{prop}
\label{prop1} When $n=2,$ any radial solution on $(a,\infty)$ to (\ref{hs3})
must be special Lagrangian.
\end{prop}

\begin{proof}
It suffices to show $C=0$. For $n=2$ the equation (\ref{hs2}) \ is simply%
\begin{equation}
\label{n=2}\theta_{r}=C\frac{1}{r}\sqrt{\frac{1+\left(  u^{\prime\prime
}\right)  ^{2}}{1+\left(  \frac{u^{\prime}}{r}\right)  ^{2}}}.
\end{equation}
Because $\theta$ stays finite, we have that
\begin{equation}
-\pi<\int_{a}^{\infty}C\frac{1}{r}\sqrt{\frac{1+\left(  u^{\prime\prime
}\right)  ^{2}}{1+\left(  \frac{u^{\prime}}{r}\right)  ^{2}}}dr<\pi.
\label{thetaprime}%
\end{equation}

We proceed by contradiction to show $C=0$. Without loss of generality, we may
assume that $C>0$, hence $\theta$ is increasing. If $C<0,$ note that the left
hand side of (\ref{hs2}) is odd in $u$, while the right hand side is even.
Thus we may replace $u$ with $-u$ to obtain $C>0.$

In view of Proposition \ref{easy proof}, we only need to consider the case
that $\lim_{r\rightarrow\infty}\theta=0$. Note that $\theta<0$ because
$\theta$ is increasing in $r$ (as $C>0)$ and $\theta$ vanishes at infinity.

It follows from
\[
\theta=\arctan(u^{\prime\prime})+\arctan\left(  \frac{u^{\prime}}{r}\right)
<0
\]
that
\[
u^{\prime\prime}+\frac{u^{\prime}}{r}<0.
\]
Now
\[
\left(  ru^{\prime}\right)  ^{\prime}<0
\]
from which it follows that for any $t_{0}$
\begin{equation}
ru^{\prime}\leq t_{0}u^{\prime}(t_{0})\text{ for }r>t_{0.} \label{uprimebound}%
\end{equation}
First note that $u^{\prime}/r$ must be unbounded. If not, the integral
expression in (\ref{thetaprime}) would be unbounded. It follows that there is
a sequence of $r_{k}$ for which $\left\vert {u^{\prime}(r_{k})}/{r_{k}%
}\right\vert \rightarrow\infty$ as $k\rightarrow\infty$. Noting
(\ref{uprimebound}) we must have ${u^{\prime}(r_{k})}/{r_{k}\rightarrow
-\infty.}$ We conclude from (\ref{uprimebound}) that there is $r_{0}$ so that
\begin{equation}
u^{\prime}\leq0, \ \ \text{ for }r>r_{0}. \label{tocontradict}%
\end{equation}
Now letting
\[
q(r)=\operatorname{arcsinh}\left(  \frac{u^{\prime}(r)}{r}\right)
\]
we conclude
\[
q(r_{k})\text{ }\rightarrow-\infty,\,\,\,\,\,\text{as $k\rightarrow\infty$. }%
\]
Computing the derivative,
\[
q^{\prime}=\frac{1}{\sqrt{1+\left(  \frac{u^{\prime}}{r}\right)  ^{2}}}\left(
\frac{u^{\prime\prime}}{r}-\frac{u^{\prime}}{r^{2}}\right)  .
\]
Now because $q$ is unbounded below, for any given large $M_{k}>0$, we can find
an interval $[a+1,s_{k}]\subset(a,\infty)$ such that
\[
\int_{a+1}^{s_{k}}q^{\prime}(r)dr=q(s_{k})-q(a+1)=-M_{k}.
\]
It follows that
\begin{equation}
\int_{a+1}^{s_{k}}\frac{1}{\sqrt{1+\left(  \frac{u^{\prime}}{r}\right)  ^{2}}%
}\frac{u^{\prime\prime}}{r}dr=-M_{k}+\int_{a+1}^{s_{k}}\frac{1}{\sqrt
{1+\left(  \frac{u^{\prime}}{r}\right)  ^{2}}}\frac{u^{\prime}}{r^{2}}dr.
\label{u"}%
\end{equation}
Now%
\begin{align}
\left\vert \int_{a+1}^{s_{k}}\frac{1}{\sqrt{1+\left(  \frac{u^{\prime}}%
{r}\right)  ^{2}}}\frac{u^{\prime\prime}}{r}dr\right\vert  &  \leq\left\vert
\int_{a+1}^{s_{k}}\frac{1}{r}\sqrt{\frac{1+\left(  u^{\prime\prime}\right)
^{2}}{1+\left(  \frac{u^{\prime}}{r}\right)  ^{2}}}dr\right\vert
\label{finite}\\
&  =\frac{1}{C}\left\vert \int_{a+1}^{s_{k}}\theta_{r}\,dr\right\vert
\nonumber\\
&  =\frac{1}{C}\left\vert \theta(s_{k})-\theta(a+1)\right\vert \nonumber\\
&  <+\infty,\nonumber
\end{align}
using \eqref{n=2}. Taking $M_{k}\rightarrow\infty$ we conclude from \eqref{u"}
and \eqref{finite} that
\[
\lim_{s_{k}\rightarrow\infty}\int_{a+1}^{s_{k}}\frac{1}{\sqrt{1+\left(
\frac{u^{\prime}}{r}\right)  ^{2}}}\frac{u^{\prime}}{r^{2}}\,dr=+\infty.
\]
Clearly, for this to happen we must have $u^{\prime}>0$ for at least a
sequence of $r\rightarrow\infty.$ But this contradicts (\ref{tocontradict}).

We conclude that $C=0$ and $\theta_{r}=0.$ Thus $u$ solves the special
Lagrangian equation.
\end{proof}

\begin{rem}
{{For $n=2$, the special Lagrangian equation can be written as
\[
\left(  1-u_{11}u_{22}+u_{12}^{2}\right)  \sin\theta+\left(  u_{11}%
+u_{22}\right)  \cos\theta=0
\]
where $u_{ij}$ stands for the second order derivative of $u$ in $x_{i},x_{j}$.
When $u$ is radial, the above equation takes the form
\[
\left(  1-\frac{u^{\prime}u^{\prime\prime}}{r}\right)  \sin\theta+\left(
u^{\prime\prime}+\frac{u^{\prime}}{r}\right)  \cos\theta=0.
\]
This can also be written as
\[
\left(  r^{2}-(u^{\prime})^{2}\right)  ^{\prime}\sin\theta+2(ru^{\prime
})^{\prime}\cos\theta=0.
\]
Completing the square and integrating}}%
\[
\int_{a}^{r}\left[  \left(  u^{^{\prime}}(t)-t\cot\theta\right)  ^{2}%
-t^{2}\csc^{2}\theta\right]  ^{\prime}dt=0
\]
we get
\[
\left(  u^{^{\prime}}(r)-r\cot\theta\right)  ^{2}-r^{2}\csc^{2}\theta=\left(
u^{^{\prime}}(a)-a\cot\theta\right)  ^{2}-a^{2}\csc^{2}\theta
\]
or%
\[
\left(  u^{^{\prime}}(r)-r\cot\theta\right)  ^{2}=\csc^{2}\theta\left[
r^{2}+\beta_{a,u,\theta}\right]
\]
for
\[
\beta_{a,u,\theta}=\sin^{2}\theta\left[  u^{^{\prime}}(a)\right]  ^{2}%
-a^{2}\sin^{2}\theta-2a\cos\theta\sin\theta.
\]
{{Thus
\[
u^{\prime}(r)=\frac{1}{\sin\theta}\left(  r\cos\theta\pm\sqrt{r^{2}%
+\beta_{a,u,\theta}}\right)  .
\]
Integrating }}%
\begin{align*}
\int_{a}^{r}u^{\prime}(t)dt &  =\frac{\cot\theta}{2}\left(  r^{2}%
-a^{2}\right)  \pm\csc\theta\int_{a}^{r}\sqrt{t^{2}+\beta_{a,u,\theta}}dt\\
&  =\frac{\cot\theta}{2}\left(  r^{2}-a^{2}\right)  +\csc\theta\frac{1}%
{2}\left(  r\sqrt{r^{2}+\beta}-a\sqrt{a^{2}+\beta}\right)  +\frac{\beta}%
{2}\csc\theta\ln\left(  \frac{r+\sqrt{r^{2}+\beta}}{a+\sqrt{a^{2}+\beta}%
}\right)  .
\end{align*}
Thus
\begin{equation}
u(r)=u(a)+\frac{\cot\theta}{2}(r^{2}-a^{2})\pm\frac{\csc\theta}{2}\left(
r\sqrt{r^{2}+\beta}-a\sqrt{a^{2}+\beta}+\beta\ln\left(  \frac{r+\sqrt
{r^{2}+\beta}}{a+\sqrt{a^{2}+\beta}}\right)  \right)  \,.\label{2dform}%
\end{equation}
Therefore, according to Proposition \ref{prop1}, all radial solutions to
\eqref{Hamstat} defined on an unbounded domain $\left\{  x\in\mathbb{R}%
^{2}:|x|\geq a\right\}  $ are given by the above formula (\ref{2dform}) explicitly.
\end{rem}



\section{Existence of non special Lagrangian solutions on unbounded domain for
$n>2$}

Observe that one instance of the equation (\ref{main}), i.e. $C=1-n$, can be
written as
\begin{equation}
F_{\lambda}(v)=0 \label{F-lambda}%
\end{equation}
where
\[
\lambda=n-1
\]
and
\begin{equation}
F_{\lambda}(v)=\frac{1}{1+\left(  v^{\prime}\right)  ^{2}}v^{\prime\prime
}+\lambda\frac{v^{\prime}r-v}{r^{2}+v^{2}}+\lambda\frac{\left(  1+\left(
v^{\prime}\right)  ^{2}\right)  ^{1/2}}{\left(  r^{2}+v^{2}\right)
^{\lambda/2}}. \label{flambda}%
\end{equation}
Note that solutions to \eqref{F-lambda} cannot be special Lagrangian since
$C\not =0$.

\begin{prop}
\label{prop41}For $\lambda\geq2,$ the equation (\ref{F-lambda}) admits
nontrivial solutions on $\left(  0,\infty\right)  .$
\end{prop}

\begin{proof}
Let
\begin{align*}
v(0)  &  =1\\
v^{\prime}(0)  &  =1.
\end{align*}
Equivalent to (\ref{F-lambda}), we are trying to solve
\begin{equation}
v^{\prime\prime}=G_{\lambda}(r,v,v^{\prime})=-\lambda\left[  \frac{\left(
1+\left(  v^{\prime}\right)  ^{2}\right)  ^{1/2}}{\left(  r^{2}+v^{2}\right)
^{\lambda/2}}+\frac{v^{\prime}r-v}{r^{2}+v^{2}}\right]  \left(  1+\left(
v^{\prime}\right)  ^{2}\right)  . \label{FirstOrderSystem}%
\end{equation}
Notice that when $v>0$ or $r>0,$ the function $G_{\lambda}(r,v,v^{\prime})$ is
smooth in terms of its arguments. Thus by the standard ODE\ theory (see for
example \cite{BR} Theorem 8, Chapter 6), by choosing $v(0)=1>0$ we guarantee
the existence of a solution on some interval $[0,T).$

Suppose that $T<\infty$ is the maximum of such $T.$ \ First we claim that
given these initial conditions we have that on $[0,T),$
\begin{align*}
v  &  \geq1\\
v^{\prime}  &  >0.
\end{align*}
To see this, we argue by contradiction. Assume it is not the case. At the
first time $r_{1}$ where $v^{\prime}=0$, it must hold that
\[
v^{\prime\prime}\leq0
\]
and \eqref{FirstOrderSystem} reads
\begin{align*}
v^{\prime\prime}  &  =-\lambda\left[  \frac{1}{\left(  r^{2}+v^{2}\right)
^{{\lambda}/{2}}}-\frac{v}{r^{2}+v^{2}}\right] \\
&  =-\lambda\frac{1}{\left(  r^{2}+v^{2}\right)  }\left[  \frac{1}{\left(
r^{2}+v^{2}\right)  ^{{\lambda}/{2}-1}}-v\right]  .
\end{align*}
Now $v>1$ and $\lambda\geq2$, so the right hand side is clearly positive, and
we have contradiction, hence proving the claim that $v^{\prime}>0$.

Next, we observe the following. Rewriting \eqref{FirstOrderSystem} as
\[
v^{\prime\prime}=-\lambda\left[  \frac{\left(  1+\left(  v^{\prime}\right)
^{2}\right)  ^{{3}/{2}}}{\left(  r^{2}+v^{2}\right)  ^{{\lambda}/{2}} }%
+\frac{v^{\prime}r}{r^{2}+v^{2}}\left(  1+\left(  v^{\prime}\right)
^{2}\right)  -\frac{v}{r^{2}+v^{2}}\left(  1+\left(  v^{\prime}\right)
^{2}\right)  \right]
\]
when $v^{\prime}>0$ we have
\begin{align*}
v^{\prime\prime}  &  \leq-\lambda\frac{v^{\prime}r}{r^{2}+v^{2}}\left(
1+\left(  v^{\prime}\right)  ^{2}\right)  +\lambda\frac{v}{r^{2}+v^{2}}\left(
1+\left(  v^{\prime}\right)  ^{2}\right) \\
&  \leq\lambda\frac{v}{r^{2}+v^{2}}+\lambda\frac{\left(  v^{\prime}\right)
^{2}v-r\left(  v^{\prime}\right)  ^{3}}{r^{2}+v^{2}}.
\end{align*}
In particular, after some small $\delta>0$ where the solution exists, (choose
$\delta<{1}/\sqrt{\lambda}$ ) we see that on $[\delta,T)$
\[
v^{\prime\prime}\leq\lambda+\lambda\frac{\left(  v^{\prime}\right)  ^{2}%
}{r^{2}+v^{2}} (v-\delta v^{\prime}).
\]
Certainly $v-\delta v^{\prime}>0$ for sufficiently small $\delta$ since
$v\geq1$ for all $r$. We claim that $v-\delta v^{\prime}>0$ for all $r$. We
argue by contradiction. Suppose this is not the case. At the first time
$r_{0}$ for which $v-\delta v^{\prime}=0$, we must\ have
\begin{align*}
v^{\prime}-\delta v^{\prime\prime}  &  \leq0\\
v  &  =\delta v^{\prime}%
\end{align*}
so at $r_{0}$
\[
v^{\prime}\leq\delta v^{\prime\prime}\leq\delta\lambda+\delta\lambda
\frac{\left(  v^{\prime}\right)  ^{2}}{r^{2}+v^{2}}(v-\delta v^{\prime
})=\delta\lambda.
\]
But $v(r_{0})>1$ since $v$ is increasing from $1$, so at $r_{0}$
\[
v^{\prime}=\frac{v}{\delta}>\frac{1}{\delta}>\delta\lambda
\]
for $\delta$ chosen small enough, which is a contradiction. Thus $v-\delta
v^{\prime}>0$ for all $r$ and we can integrate%
\[
\frac{v^{\prime}}{v}<\frac{1}{\delta}%
\]
that is
\[
v(r)\leq v(\delta)e^{\frac{1}{\delta}r}.
\]
Thus on the interval $[\delta,T)$ we have
\[
v^{\prime}\leq\frac{1}{\delta}v\leq\frac{1}{\delta}v(\delta)e^{\frac{T}%
{\delta}}\leq C_{1}
\]
This implies that $v^{\prime}$ is bounded.

It is easy to see that as long as $v>1$ and $0<v^{\prime}<C_{1}$ the right
hand side $G_{\lambda}(r,v,v^{\prime})$ of the ODE (\ref{FirstOrderSystem}) is
smooth in its arguments on any fixed bounded interval $[0,T)$, so in
particular, the solution must be well-defined at $T.$ \ Again (c.f. \cite{BR})
we may extend this to $[T,T+$ $\tau)$ for some $\tau>0$. \ It follows that for
$T<\infty,$ the interval $[0,T)$ cannot be the maximal domain of existence.
\end{proof}

\textbf{Some explicit solutions and geometric examples.} For dimensions 3 and
4 we can write down the following explicit solutions to \eqref{Hamstat},
respectively
\begin{align*}
u(r)  &  =r\\
u(r)  &  =\ln r.
\end{align*}

For the solution $u=r,$ and $n=3,$ consider the gradient graph
\[
\Gamma=\left(  x,Du(x)\right)  =\left\{  \left(  x,\frac{x}{|x|}\right)
:x\in\mathbb{B}^{3}\backslash\{0\}\right\}  \subset{\mathbb{R}}^{3}%
\oplus{\mathbb{R}}^{3}%
\]
over the unit ball (or ball of any size) in $\mathbb{R}^{3}$. While the
function $Du(x)=$ $\frac{x}{|x|}$ has an isolated non-continuous singularity
at $0\in\mathbb{B}^{3}$, the submanifold is a smooth embedded submanifold with
a smooth boundary consisting of disjoint copies of $\mathbb{S}^{2}.$ \ A
smooth embedding
\[
\mathbb{S}^{2}\times\lbrack0,1]\rightarrow\bar{\Gamma}%
\]
becomes obvious when writing
\[
\bar{\Gamma}=\left\{  (sV,V):V\in\mathbb{S}^{2},0\leq s\leq1\right\}  .
\]
This topological cylinder cannot be isotopic (Hamiltonian or otherwise) to the
graph of a smooth function over the unit ball.

Inverting $\bar{\Gamma}$ through the origin in ${\mathbb{R}}^{3}$ and gluing,
we can obtain a nongraphical smooth Hamiltonian stationary manifold%
\[
\Gamma_{0}=\left\{  (sV,V):V\in\mathbb{S}^{2},-1<s<1\right\}  .
\]

\bigskip

\section{Point singularities for radial solutions to Hamiltonian stationary
equations}

Motivated by the examples at the end of last section, we consider the
alternative possibility, that $\lim_{r\rightarrow0}u^{\prime}(r)=0$. In this
case the gradient graph is homeomorphic to the domain of the function.

To begin we establish an analytic result for later application:

\begin{lem}
\label{technical} Suppose that $v$ is a positive differentiable function on
$\left(  0,2R\right)  $ and continuous at $0$. Suppose that there is a
sequence of decreasing positive numbers $r_{k}\rightarrow0$ as $k\rightarrow
\infty$ and a constant $c>0$ such that both
\begin{align}
v(r_{k})  &  \geq cr_{k},\label{CC1}\\
v(r_{k})  &  =\max_{r\in\lbrack0,r_{k}]}v(r). \label{CC2}%
\end{align}
If
\begin{equation}
\lim_{\varepsilon\rightarrow0^{+}}\int_{\varepsilon}^{R}\sqrt{\frac{1+\left(
v^{\prime}\right)  ^{2}}{r^{2}+v^{2}}}\,dr<\infty\label{CC3}%
\end{equation}
then
\[
\lim_{r\rightarrow0}v(r)=\delta>0.
\]

\end{lem}

\begin{proof}
For the sequence of decreasing positive numbers $r_{k}\rightarrow0$ given by
the assumption,
\begin{align}
\int_{r_{k}}^{r_{k-1}}\sqrt{\frac{1+\left(  v^{{\prime}}\right)  ^{2}}%
{r^{2}+v^{2}}}dr &  \geq\int_{r_{k}}^{r_{k-1}}\frac{v^{{\prime}}}{\sqrt
{r^{2}+v^{2}}}dr\label{use}\\
&  \geq\frac{1}{\max\sqrt{r^{2}+v^{2}}}\int_{r_{k}}^{r_{k-1}}v^{{\prime}%
}dr\nonumber\\
&  =\frac{1}{\sqrt{r_{k-1}^{2}+v^{2}(r_{k-1})}}\left[  v(r_{k-1}%
)-v(r_{k})\right]  \nonumber
\end{align}
using that%
\[
v(r_{k-1})=\max_{r\in\lbrack0,r_{k}]}v(r)\geq\max_{r\in\lbrack r_{k},r_{k-1}%
]}v(r).
\]
Now using
\[
v(r_{k})\geq cr_{k},
\]
we have
\begin{align*}
\sqrt{r_{k-1}^{2}+v^{2}(r_{k-1})} &  \leq\sqrt{\frac{v^{2}(r_{k-1})}{c^{2}%
}+v^{2}(r_{k-1})}\\
&  =\sqrt{\frac{1}{c^{2}}+1}\,v(r_{k-1})
\end{align*}
thus
\[
\frac{1}{\sqrt{r_{k-1}^{2}+v^{2}(r_{k-1})}}\geq\frac{c}{\sqrt{1+c^{2}}}%
\frac{1}{v(r_{k-1})}.
\]
It then follows from \eqref{use}
\begin{align}
\int_{r_{k}}^{r_{k-1}}\sqrt{\frac{1+\left(  v^{{\prime}}\right)  ^{2}}%
{r^{2}+v^{2}}}dr &  \geq\frac{c}{\sqrt{1+c^{2}}}\frac{1}{v(r_{k-1})}\left[
v(r_{k-1})-v(r_{k})\right]  \label{use1}\\
&  \geq\frac{c}{\sqrt{1+c^{2}}}\left[  1-q_{k}\right]  \nonumber
\end{align}
where
\[
q_{k}=\frac{v(r_{k})}{v(r_{k-1})}\leq1
\]
as $r_{k}<r_{k-1}$ from (\ref{CC2}). Now let
\[
a_{k}=1-q_{k}%
\]
so that $0\leq a_{k}<1$, and then sum \eqref{use1} over $k$
\[
\lim_{k\rightarrow\infty}\int_{r_{k}}^{r_{0}}\sqrt{\frac{1+\left(
v^{^{\prime}}\right)  ^{2}}{r^{2}+v^{2}}}dr\geq\frac{c}{\sqrt{1+c^{2}}}%
\sum_{k=1}^{\infty}a_{k}.
\]
Thus
\[
\sum_{k=1}^{\infty}a_{k}\leq\frac{\sqrt{1+c^{2}}}{c}\int_{0}^{r_{0}}%
\sqrt{\frac{1+\left(  v^{^{\prime}}\right)  ^{2}}{r^{2}+v^{2}}}dr<\infty
\]
by (\ref{CC3}). \ In particular,
\[
\lim_{k\rightarrow\infty}a_{k}=0.
\]
Recall that for a sequence $q_{k}$ of positive numbers, $\prod_{k=1}^{\infty
}q_{k}$ converges if and only if $\sum_{k=1}^{\infty}\ln q_{k}$ converges.
Now
\[
\sum_{k=1}^{\infty}\ln q_{k}=\sum_{k=1}^{\infty}\ln(1-a_{k}).
\]
For large enough $K_{0}$ there exists a constant $C$ so that
\[
\left\vert \sum_{k>K_{0}}\ln(1-a_{k})\right\vert \leq C\left\vert
\sum_{k>K_{0}}a_{k}\right\vert <\infty.
\]
Thus, we may assume
\[
\sum_{k=1}^{\infty}\ln q_{k}=q>-\infty.
\]
Then
\[
\prod_{k=1}^{\infty}q_{k}=e^{q}\not =0.
\]
But
\[
q_{k}q_{k-1}\cdots q_{1}=\frac{v(r_{k})}{v(r_{k-1})}\frac{v(r_{k-1}%
)}{v(r_{k-2})}\cdots\frac{v(r_{2})}{v(r_{1})}=\frac{v(r_{k})}{v(r_{1})}.
\]
Hence
\[
\lim_{k\rightarrow\infty}v(r_{k})=e^{q}v(r_{1})>0.
\]
This concludes the proof of the lemma.
\end{proof}

\begin{cor}
\label{initial} Suppose that $v$ is a solution of (\ref{main}) on $(0,R)$. If
$C\neq0$, then $\lim_{r\rightarrow0}v(r)\neq0.$
\end{cor}

\begin{proof}
Suppose not. \ Assume that $v$ is a solution and $\lim_{r\rightarrow0}v(r)=0.$
For $0<\varepsilon<R$, integrating \eqref{hs2} leads to
\begin{align*}
\left\vert \theta(R)-\theta(\varepsilon)\right\vert  &  =|C|\int_{\varepsilon
}^{R}\sqrt{\frac{1+\left(  v^{\prime}\right)  ^{2}}{\left(  r^{2}%
+v^{2}\right)  ^{n-1}}}dr\\
&  =\left\vert C\right\vert \int_{\varepsilon}^{R}\frac{1}{\left(  r^{2}%
+v^{2}\right)  ^{\frac{n-2}{2}}}\sqrt{\frac{1+\left(  v^{\prime}\right)  ^{2}%
}{r^{2}+v^{2}}}dr\\
&  \geq\min_{r\in(\varepsilon,R)}\frac{1}{\left(  r^{2}+v^{2}\right)
^{\frac{n-2}{2}}}\left\vert C\right\vert \int_{\varepsilon}^{R}\sqrt
{\frac{1+\left(  v^{\prime}\right)  ^{2}}{r^{2}+v^{2}}}dr.
\end{align*}
Now assuming $\lim_{r\rightarrow0}v(r)=0$, $v$ is bounded, and we have
\begin{equation}
\lim_{\varepsilon\rightarrow0^{+}}\int_{\varepsilon}^{R}\sqrt{\frac{1+\left(
v^{\prime}\right)  ^{2}}{r^{2}+v^{2}}}dr\leq C_{R}<\infty
.\label{finite integral}%
\end{equation}

Next, if there is constant $t_{0}>0$ so that $|v(r)|\leq c r$ for all
$r<t_{0}$, \eqref{finite integral} cannot hold. Thus, there is a sequence of
decreasing positive numbers $r_{k}\to0$ such that
\begin{equation}
\label{r_k}|v(r_{k})|\geq c r_{k}.
\end{equation}

Now we show that $v$ does not frequently change sign. Using the odd and even
properties of the equation (\ref{hs2}), we may assume that $C>0.$ In this
case, if $v$ has a nonnegative local maximum, then at the point where the
maximum is attained, we have
\begin{align*}
0\geq\frac{1}{1+\left(  v^{\prime}\right)  ^{2}}v^{\prime\prime}  &  =-\left(
n-1\right)  \frac{v^{\prime}r-v}{r^{2}+v^{2}}+C\,\frac{(1+\left(  v^{\prime
}\right)  ^{2})^{\frac{1}{2}}}{\left(  r^{2}+v^{2}\right)  ^{\frac{n-1}{2}}}\\
&  =\left(  n-1\right)  \frac{v}{r^{2}+v^{2}}+C\,\frac{1}{\left(  r^{2}%
+v^{2}\right)  ^{\frac{n-1}{2}}}>0
\end{align*}
a contradiction. In particular, if $v^{\prime}(a)>0$ and $v(a)>0$, then we
must have $v^{\prime}(r)>0$ for all $r>a.$ It follows that $v(r)=0$ can occur
no more than once (not counting at $r=0$). We conclude that there is an
interval $(0,t_{1})$ where $v$ is either positive or negative.

If $v>0$ on $(0,t_{1})$, we can apply Lemma \ref{technical} to conclude that
$\lim_{r\rightarrow0}v(r)\not =0$, recalling (\ref{finite integral}) and
(\ref{r_k}).

If $v<0$ on $(0,t_{1})$, we note that if there exists a positive number
$t_{2}<t_{1}$ such that $v^{\prime}\leq0$ on $(0,t_{2})$, then Lemma
\ref{technical} applies to $-v$ and yields the contradiction. So we assume
that there is a sequence of decreasing positive numbers $s_{k}^{(1)}%
\rightarrow0$ with $v^{\prime}(s_{k}^{(1)})>0$. Note that there also exists
another decreasing sequence $s_{k}^{(2)}\rightarrow0$ where $v^{\prime}%
(s_{k}^{(2)})\leq0$, otherwise $v$ would be positive near $0$ as we are
assuming $\lim_{r\rightarrow0}v(r)=0$, this would contradict $v<0$ on
$(0,t_{1}).$ Then, there would exist a decreasing sequence $t_{k}\rightarrow0$
where $v$ attains local maximum. We now treat the the two cases $n\geq3$ and
$n=2$ separately. When $n\geq3$
\[
0\geq v^{\prime\prime}(t_{k})=(n-1)\frac{v(t_{k})}{t_{k}^{2}+v^{2}(t_{k}%
)}+\frac{C}{(t_{k}^{2}+v^{2}(t_{k}))^{\frac{n-1}{2}}}%
\]
which gives a contradiction since $v(t_{k})\rightarrow0$ when $t_{k}%
\rightarrow0$, recalling again that $C>0.$

We consider next the case $n=2.$ \ In this case, we may find some $r_{0}$ with
$v^{\prime}(r_{0})>0.$ $\ $Choose an interval $(s_{0},t_{0})$ so that
$r_{0}\in(s_{0},t_{0})$ and
\begin{align*}
v^{\prime}(s_{0})  &  =0\\
v^{\prime}(t_{0})  &  =0\\
v^{\prime}(r)  &  >0\text{ on }(s_{0},t_{0})
\end{align*}
\newline At~$t_{0}$,
\[
0\geq\left(  t_{0}^{2}+v^{2}(t_{0})\right)  v^{\prime\prime}(t_{0}%
)=v(t_{0})+C\left(  t_{0}^{2}+v^{2}(t_{0})\right)  ^{\frac{1}{2}}.
\]
This yields that
\[
v^{2}(t_{0})\geq C^{2}\left(  t_{0}^{2}+v^{2}(t_{0})\right)
\]
so clearly $C<1$ and%
\[
\left\vert v(t_{0})\right\vert \geq\frac{C}{\sqrt{1-C^{2}}}t_{0}.
\]
Now at $s_{0}<r_{0}$ we will have
\begin{equation}
\left\vert v(t_{0})\right\vert =-v(t_{0})<-v(s_{0})=\left\vert v(s_{0}%
)\right\vert \label{ssc}%
\end{equation}
and
\[
0\leq\left(  s_{0}^{2}+v^{2}(s_{0})\right)  v^{\prime\prime}(s_{0}%
)=v(s_{0})+C\left(  s_{0}^{2}+v^{2}(s_{0})\right)  ^{\frac{1}{2}}.
\]
The latter inequality gives us
\[
\left\vert v(s_{0})\right\vert \leq\frac{C}{\sqrt{1-C^{2}}}s_{0}<\frac
{C}{\sqrt{1-C^{2}}}t_{0}\leq\left\vert v(t_{0})\right\vert
\]
which is clearly a contradiction of (\ref{ssc}).
\end{proof}

We now prove Theorem \ref{two}. \ 

\begin{proof}
Suppose that $u$ is a radial solution to (\ref{Hamstat}) on a neighborhood
$\mathbb{B}_{R}(0)$ with
\[
\lim_{r\rightarrow0}u^{\prime}(r)=0.
\]

Since $\lim_{r\rightarrow0}u^{\prime}(r)=0$, the contrapositive of Corollary
\ref{initial} asserts that $\theta$ is constant, so the gradient graph of $u$
is a calibrated submanifold which is smooth away from the origin and
continuous everywhere. For any $\rho\in(0,R]$, define
\[
u^{(\rho)}(x)=\frac{u^{\prime}(\rho)}{2\rho}|x|^{2}.
\]
Notice that on $\partial B_{\rho}(0)$ we have
\[
Du^{(\rho)}(x)=u^{\prime}(\rho)\frac{x}{\rho}=Du(x)
\]
and that because the latter is a continuous graph over the whole domain, these
gradient graphs will be isotopic. Now we may use the special Lagrangian
calibration argument (\cite[Section III]{HL}) to compare the volumes of the
graphs
\[
\Gamma=\left\{  (x,Du(x)):\ x\in B_{\rho}(0)\right\}
\]
and
\[
\Gamma^{(\rho)}=\left\{  (x,Du^{(\rho)}(x)):\ x\in B_{\rho}(0)\right\}  .
\]
First, let
\[
\phi=\operatorname{Re}\left(  e^{-\sqrt{-1}\theta}dz_{1}\wedge...\wedge
dz_{n}\right)  .
\]
Note that the closed $n$-form $\phi$ calibrates $\Gamma$. By Stokes' Theorem,
as $d\phi=0$, we have
\[
\mbox{Vol}(\Gamma)=\int_{\Gamma}\phi=\int_{\Gamma^{(\rho)}}\phi\leq
\mbox{Vol}(\Gamma^{(\rho)}).
\]
On the other hand, letting%
\[
\phi^{(\rho)}=\operatorname{Re}\left(  e^{-\sqrt{-1}\theta^{(\rho)}}%
dz_{1}\wedge...\wedge dz_{n}\right)
\]
where $\theta_{\rho}$ is the constant given by
\[
\theta^{(\rho)}=n\arctan\left(  \frac{u^{\prime}(\rho)}{\rho}\right)
\]
and noting that the closed $n$-form $\phi^{(\rho)}$ calibrates $\Gamma
^{(\rho)}$. We conclude that
\[
\mbox{Vol}(\Gamma^{(\rho)})=\int_{\Gamma^{(\rho)}}\phi^{(\rho)}=\int_{\Gamma
}\phi^{(\rho)}\leq\mbox{Vol}(\Gamma).
\]
It then follows that
\[
\mbox{Vol}(\Gamma^{(\rho)})=\mbox{Vol}(\Gamma)=\int_{\Gamma}\phi
\]
and
\[
\theta=\theta^{(\rho)}.
\]
This implies that
\[
u^{\prime}(\rho)=\rho\tan\left(  \frac{\theta}{n}\right)  \,
\]
for any $\rho\in(0,R]$, therefore $u$ is quadratic since $\theta$ is a constant.

Next, we assume that
\[
\lim_{r\rightarrow0}u^{\prime}(r)=c\neq0.
\]
Letting $u(0)=c,$ we have a map (extending (\ref{embedding})),%
\[
F:[0,\infty)\times\mathbb{S}^{n-1}\rightarrow\mathbb{R}^{n}\times
\mathbb{R}^{n}%
\]
defined by
\begin{equation}
F(r,\xi)=(r\xi,u^{\prime}(r)\,\xi)\in\mathbb{R}^{n}\times\mathbb{R}^{n}.
\end{equation}
It is not hard to see that this immersion is proper and injective, so must be
an embedding.
\end{proof}

\bibliographystyle{amsalpha}
\bibliography{hamstat1}

\end{document}